\documentclass[12pt]{article}

\usepackage{amsmath}
\usepackage{amssymb}
\usepackage{amsthm}
\usepackage{latexsym}
\usepackage{comment}
\usepackage{authblk}
\usepackage{graphicx}

\DeclareMathOperator{\Wr}{Wr}
\DeclareMathOperator{\Real}{Re}
\DeclareMathOperator{\Imag}{Im}
\renewcommand{\Re}{\Real}
\renewcommand{\Im}{\Imag}

\newcommand{\cH}{\mathcal{H}}

\newtheorem{theorem}{Theorem}[section]

\newtheorem{lemma}[theorem]{Lemma}

\newtheorem{proposition}[theorem]{Proposition}
\newtheorem{corollary}[theorem]{Corollary}
\newtheorem{conjecture}[theorem]{Conjecture}

\numberwithin{equation}{section}

\begin{document}

\author{A.B.J. Kuijlaars\footnote{Department of Mathematics, Katholieke Universiteit Leuven, Belgium, email: arno.kuijlaars@wis.kuleuven.be}
\, and  R. Milson\footnote{Department of Mathematics and Statistics, Dalhousie University, Halifax, Canada,
email: rmilson@dal.ca}}
\title{Zeros of exceptional Hermite polynomials}
\date{\today}

\maketitle
\begin{abstract} 
  We study the zeros of exceptional Hermite polynomials associated
  with an even partition $\lambda$. We prove several conjectures
  regarding the asymptotic behaviour of both the regular (real) and the
  exceptional (complex) zeros.  The real zeros are distributed as the
  zeros of usual Hermite polynomials and, after contracting by a
  factor $\sqrt{2n}$, we prove that they follow the semi-circle law.
  The non-real zeros tend to the zeros of the generalized Hermite
  polynomial $H_{\lambda}$, provided that these zeros are simple. It
  was conjectured by Veselov that the zeros of generalized Hermite
  polynomials are always simple, except possibly for the zero at the
  origin, but this conjecture remains open.
\end{abstract}


\section{Introduction}
The field of classical orthogonal polynomials is essentially the study
of Sturm-Liouville problems with polynomial solutions.  Indeed, by the
well-known theorem of Bochner, if we assume that a Sturm-Liouville
problem admits an eigenpolynomial of \emph{every} degree, the we
arrive at the well-known families of Hermite, Laguerre, Jacobi, and
Bessel (if signed weights are allowed).  Exceptional orthogonal
polynomials arise when we consider Sturm-Liouville problems with
polynomial eigenfunctions, but allow a finite number of degrees to be
missing from the corresponding degree sequence.  For background on
exceptional orthogonal polynomials and exact solutions in quantum
mechanics, see \cite{GKM2,Sa14}. For recent developments in the area
of exceptional Hermite polynomials, see \cite{duran,GGM}.

The study of the zeros of exceptional orthogonal polynomials has
attracted some recent interest.  Preliminary results indicate that
there are strong parallels with the behaviour of zeros of classical
orthogonal polynomials. For example, it is possible to describe the
zeros of certain exceptional OP using an electrostatic interpretation
\cite{dimlun14,hor14}.  Asymptotic behaviour of the zeros of 1-step
Laguerre and Jacobi exceptional polynomials as the degree $n$ goes to
infinity was considered in \cite{GMM,HoSasaki,LLMS}. 

All known families of exceptional OP have a weight of the form
\begin{equation}
  \label{eq:xop-weight}
   W(x) = \frac{W_0(x)}{\eta(x)^2} 
\end{equation}
where $W_0(x)$ is a classical OP weight and where $\eta(x)$ is a
certain polynomial whose degree is equal to the number of gaps in the
XOP degree sequence, and which doesn't vanish on the domain of
orthogonality. It has recently been shown \cite{MGGM} that the weight
indeed takes the above form for every exceptional OP family.

The zeros of exceptional orthogonal polynomials are
divided into two groups according to whether they lie in the domain of
orthogonality. The \emph{regular} zeros are of this type and enjoy the
usual intertwining behaviour common to solutions of all
Sturm-Liouville problems.  All other types of zeros are called
\emph{exceptional zeros}.  For sufficiently high degree $n$, the
number of exceptional zeros is precisely equal to the degree of
$\eta(x)$.  Based on all extant investigations of the asymptotics of
the zeros of exceptional OP it is reasonable to formulate the
following.
\begin{conjecture} \label{conj:xzeros}
  The regular zeros of exceptional OP have the same asymptotic
  behaviour as the zeros of their classical counterpart.  The
  exceptional zeros converge to the zeros of the denominator
  polynomial $\eta(x)$.
\end{conjecture}

The first part of this conjecture admits two useful interpretations.
The simplest interpretation is that after suitable normalization, the
$k$-th regular zero of the exceptional polynomials converges to the
$k$-th zero of their classical counterpart.  Such asymptotic behaviour
can be proved by means of Mehler-Heine type theorems, as was done for
the case of certain exceptional Laguerre and Jacobi polynomials in \cite{GMM}
and \cite{LLMS}.

However, there is another way to formulate this conjecture.  It is
well known that as the degree goes to infinity, the counting measure
for the zeros of classical orthogonal polynomials, suitably
normalized, tends to a certain equilibrium measure.  The conjecture
then is that the normalized counting measure of the regular zeros of
exceptional orthogonal polynomials converges to the same equilibrium
measure.

In this paper, we partially prove the conjecture for the class of
exceptional Hermite polynomials \cite{GGM}.  Exceptional Hermite polynomials are
Wronskians of $r+1$ classical Hermite polynomials, where $r$ of the
polynomials are fixed and where the degree of the last polynomial
varies.  The precise statement of the results is given in Section
\ref{sect:XHermite} after some necessary definitions.   The proofs for
the real zeros are given in Sections \ref{sect:MehlerHeine} and  \ref{sect:reg-zeros}, and for the
exceptional zeros in Section \ref{sect:exc-zeros}.

\section{Exceptional Hermite polynomials} \label{sect:XHermite}

A partition $\lambda$ of length $r$ is a finite weakly decreasing sequence of positive integers
\[ \lambda = (\lambda_1 \geq \lambda_2 \geq \cdots \geq \lambda_r \geq 1), \] 
The number $|\lambda| = \sum_{j=1}^r \lambda_j$ is the size of
the partition.  The generalized Hermite polynomial associated with $\lambda$ is
\begin{equation}
  \label{eq:Hlambda}
  H_{\lambda} = \Wr \left[ H_{\lambda_r}, \ldots, H_{\lambda_2 + r-2}, H_{\lambda_1 + r-1} \right] 
\end{equation}
where $\Wr$ denotes the Wronskian determinant. The degree of $H_{\lambda}$ is
\[ \deg H_{\lambda} = |\lambda|. \] 
Alternatively, we may also index the generalized Hermite polynomial by means of the 
strictly decreasing sequence $k_1 > k_2 > \cdots > k_r$ where $k_j = \lambda_j+r-j$,
and then 
$H_{\lambda} = \Wr \left[H_{k_r}, \ldots, H_{k_2}, H_{k_1} \right]$.

We call $\lambda$ an even (or double) partition if $r$ is even and
$\lambda_{2k-1} = \lambda_{2k}$ for $k=1,2, \ldots, r/2$. In that case
it is known that $H_{\lambda}$ has no zeros on the real line  
\cite{Adler,GGM,Krein}, and hence
\begin{equation}
  \label{eq:Wlambdadef}
  W_{\lambda}(x) = \frac{e^{-x^2}}{(H_\lambda(x))^2}, \qquad x\in (-\infty,\infty), 
\end{equation}
is a well-defined weight function on the real line.

For a fixed partition $\lambda$, set
\begin{equation} \label{eq:degseq}
	\mathbb N_{\lambda} = \{ n \geq |\lambda|-r \mid n \neq |\lambda|+\lambda_j-j \text{ for } j =1, \ldots, r\},
	\end{equation}
and, for $n \geq |\lambda|- r$ set
\begin{equation}
  \label{eq:Pndef}
  P_n = \Wr \left[ H_{\lambda_r},
    H_{\lambda_{r-1}+1},\ldots,H_{\lambda_2 + r-2},    H_{\lambda_1 + r-1}, H_{n-|\lambda|+r} \right].
\end{equation}
Then $P_n$ is a polynomial of degree $n$ if $n\in \mathbb N_\lambda$, and 
$P_n$ vanishes identically otherwise.  For this reason we call
$\mathbb N_{\lambda}$ the degree sequence for $\lambda$. The complement
$\mathbb N_0 \setminus \mathbb N_{\lambda}$ has the forbidden
degrees. Their number are exactly $|\lambda|$. The largest forbidden
degree is $|\lambda| + \lambda_1 -1$.

It can be shown \cite[Proposition 5.2]{GGM} that the above polynomial
$P_n$ satisfies the second-order differential equation 
\begin{equation}
  \label{eq:PnDE}
 P_n''(x)-2\left(x+\frac{H_\lambda'}{H_\lambda} \right) P_n'(x)+
\left(\frac{H_\lambda''}{H_\lambda}+ 2 x
  \frac{H_\lambda'}{H_\lambda}+2n-2|\lambda| \right) P_n(x) = 0.
\end{equation}
Equivalently, we can say that $\varphi_n(x) = \frac{P_n(x)}{H_{\lambda}(x)} e^{-\frac{1}{2} x^2}$ is
an eigenfunction of the differential operator
\begin{equation} \label{eq:Doplambda} 
	- y'' + \left(x^2 - 2 \frac{d^2}{dx^2}  \log H_{\lambda}\right) y 
	\end{equation}
with eigenvalue $2n - 2 |\lambda|+1$.
As a consequence, we can express the following weighted product as a perfect
derivative: 
\begin{equation} \label{eq:PnPm}
	P_n(x) P_m(x) W_{\lambda}(x) = \frac{d}{dx} \left[ \frac{P_n(x)
    P_m'(x) - P_n'(x) P_m(x)}{2(n-m)} W_{\lambda}(x) \right], \quad n\neq m. \end{equation}
If $\lambda$ is an even partition, then \eqref{eq:PnPm} implies  the
remarkable orthogonality property
\begin{equation}
  \label{eq:Pnortho}
  \int_{-\infty}^{\infty} P_n(x) P_m(x) W_{\lambda}(x) dx = 0, \qquad
  n,m  \in \mathbb N_{\lambda}, \ n \neq m.   
\end{equation}
It is also known \cite{GMM,duran} that the closure of the linear span of 
$\{ P_n \}_{n\in {\mathbb  N}_\lambda}$ is dense  in the Hilbert space
$\mathrm L^2(\mathbb R,W_\lambda)$.  Because of this fact and because
of \eqref{eq:PnDE} and \eqref{eq:Pnortho}, the polynomials $P_n$, $n \in {\mathbb N}_\lambda$ 
are called exceptional Hermite polynomials.

By the Sturm oscillation theorem,  there are
\[ n-|\lambda| + \left| \{ j = 1, \ldots, r \mid \lambda_j-j \geq n - |\lambda|   \}\right| \] 
real zeros of $P_n$; we call these the \emph{regular}
zeros.  In particular, for $n\geq |\lambda|+\lambda_1$,  there are
exactly $n-|\lambda|$ regular (real) zeros.  For $n\geq |\lambda|+\lambda_1$, let
\begin{equation} \label{eq:realzeros} 
	x_{1,n}< x_{2,n} < \cdots < x_{n-|\lambda|,n} 
	\end{equation} denote the regular
zeros of $P_n$ arranged in ascending order.  Since $H_\lambda$ is an
even polynomial, equation \eqref{eq:PnDE} has parity invariance, and
consequently $P_n(x)$ has the same parity as $n$. 
The remaining $|\lambda|$ non-real zeros are called
\emph{exceptional} zeros; we will denote them as 
\begin{equation} \label{eq:Xzeros}
	z_{1,n}, z_{2,n}, \ldots, z_{|\lambda|,n}. 
\end{equation}  

We are now able to state our main results. Throughout the rest of the paper, $\lambda$ is a fixed
even partition. Our first result is the scaling limit of the central zeros of $P_n$.
\begin{theorem}
  \label{thm:zerospacing}
  For every fixed  $k \in \mathbb Z$ we have
  \begin{align}  \label{eq:evenzeros}
    \lim_{n \to \infty} 2 \sqrt{n}\, x_{k+n+1- \frac{|\lambda|}{2},2n} & =
    \frac{\pi}{2} + k \pi, \\ 
		\label{eq:oddzeros}
    \lim_{n \to \infty} 2 \sqrt{n}\, x_{k+n+1-\frac{|\lambda|}{2},2n+1} & = k \pi.
  \end{align}
\end{theorem}
\noindent
Theorem \ref{thm:zerospacing} will follow from a Mehler-Heine asymptotic formula for $P_n$.

The next result is the weak scaling limit of the counting measure of
the real zeros.
\begin{theorem}
\label{thm:semicircle}
For every bounded continuous function $f$ on $\mathbb R$ we have
\begin{equation} \label{eq:sclimit1}
   \lim_{n\to \infty} \frac{1}{n} \sum_{j=1}^{n-|\lambda|} f \left(
  \frac{x_{j,n}}{\sqrt{2n}}\right) = \frac{2}{\pi} \int_{-1}^1 f(x)
\sqrt{1-x^2} dx. \end{equation}
\end{theorem}
\noindent
Theorem \ref{thm:semicircle} says that the normalized counting measure of the real
zeros, when scaled by a factor $\sqrt{2n}$, tends weakly to the measure with density $\frac{2}{\pi} \sqrt{1-x^2}$
on $[-1,1]$. This measure is known as the semicircle law. The corresponding result for the zeros of 
the Hermite polynomial $H_n$ is well-known, see e.g.\ \cite{deift}.

Our final result deals with the non-real zeros. We use $z_1, \ldots, z_{|\lambda|}$ to
denote the zeros of $H_{\lambda}$.

\begin{theorem} \label{thm:xzeros}
 Let $z_j$ be a simple zero of $H_{\lambda}$.
 Then there exists a constant $C > 0$ such that for all $n$ large enough,
there is a zero $z_{k,n}$ of $P_n$ such that
\[ |z_j - z_{k,n}| \leq \frac{C}{\sqrt{n}}. \]
\end{theorem}

Theorem \ref{thm:xzeros} proves the statement in Conjecture
\ref{conj:xzeros} on the exceptional zeros in the case that all zeros
of $H_{\lambda}$ are simple.  In that case, we can relabel the
exceptional zeros \eqref{eq:Xzeros} in such a way that for every $j =
1, \ldots, |\lambda|$,
\[ z_{j,n} = z_j + O\left(\frac{1}{\sqrt{n}}\right) \quad \text{as } n \to \infty, \]
and then clearly the sequence $(z_{j,n})_n$ tends to $z_j$ for every $j$.

The condition of simple zeros may not be too restrictive, since it
is  actually believed that all non-real zeros of a generalized Hermite polynomial should be simple.
In fact, Alexander Veselov made the following  conjecture, that is quoted in \cite{FHV}:
\begin{conjecture} \label{conj:simplezeros}
For any partition $\lambda$ the 
zeros of $H_{\lambda}(z)$ are simple, except possibly for the
zero at $z=0$.
\end{conjecture}

Our results are confirmed by numerical experiments. Figure \ref{fig:plotzeros} shows the
twelve zeros of $H_{\lambda}$ where $\lambda = (4,4,2,2)$, which are simple and non-real.
The figure also shows the zeros of the exceptional Hermite polynomial $P_{40}$ of degree $40$.
It has $28$ real zeros and $12$ non-real zeros that are close to the zeros of $H_{\lambda}$
as predicted by Theorem \ref{thm:xzeros}.

\begin{figure}
\centering
\includegraphics[scale=0.5]{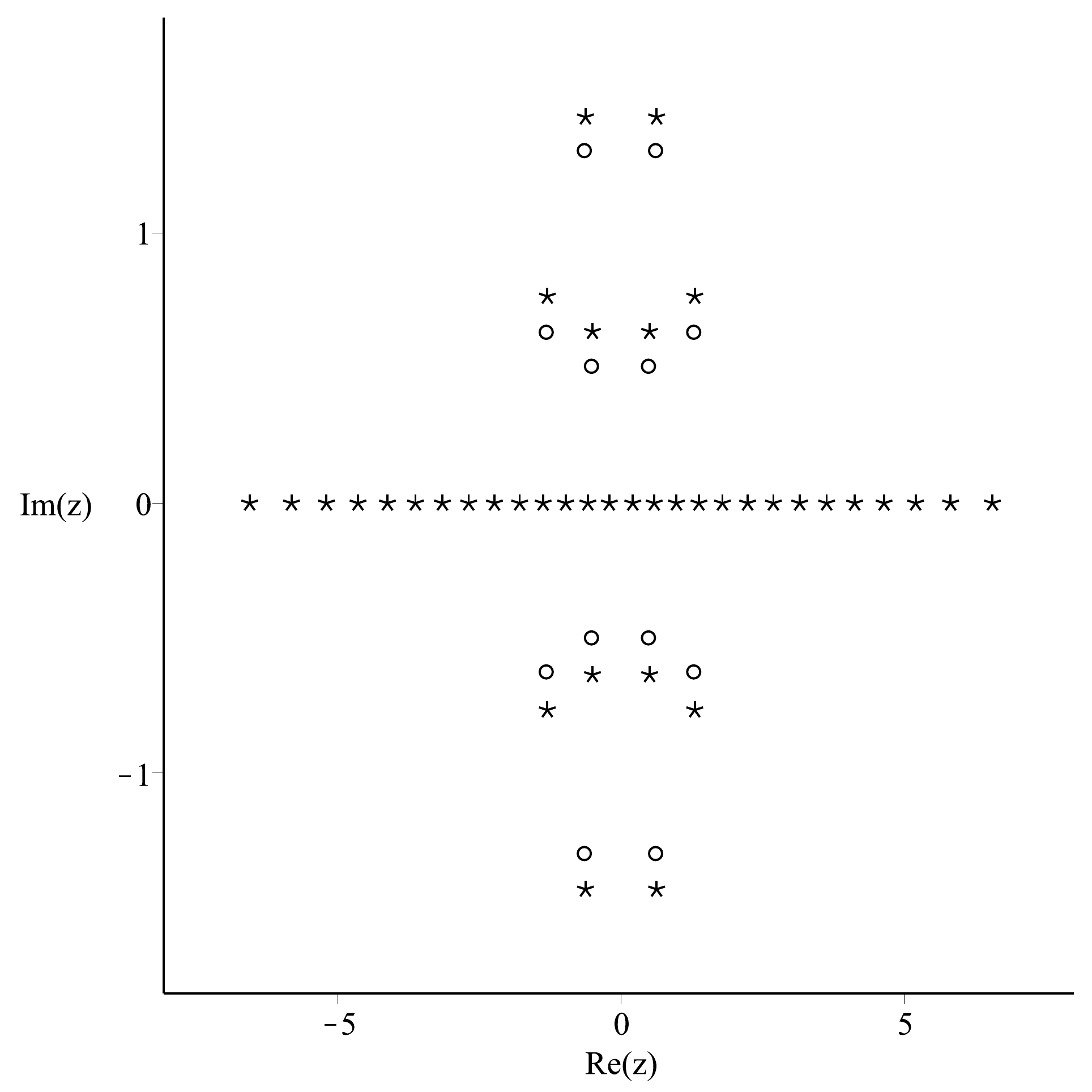}
\caption{Plot of the zeros of $H_{\lambda}$ with $\lambda = (4,4,2,2)$ (open circles)
together with the zeros of the corresponding exceptional Hermite polynomial of
degree $40$ (stars). Each zero of $H_{\lambda}$ attracts exactly one non-real zero of $P_n$ as $n \to \infty$.}
\label{fig:plotzeros}
\end{figure}


\section{Proof of Theorem \ref{thm:zerospacing} and a Mehler-Heine formula} \label{sect:MehlerHeine}
In this section we prove Theorem \ref{thm:zerospacing} by means of a Mehler-Heine
formula for the polynomials $P_n$, which may be of interest in itself.
This is a generalization of the classical Mehler-Heine formula for Hermite polynomials
\begin{equation} \label{eq:HermiteMH}
\begin{aligned} 
  \frac{(-1)^n \sqrt{n\pi}}{2^{2n} n!}\,
  H_{2n}\!\left(\frac{x}{2\sqrt{n}}\right)  &\rightrightarrows 
  \cos x,  \\
  \frac{(-1)^n\sqrt{\pi}}{2^{2n+1} n!}\,
  H_{2n+1}\!\left(\frac{x}{2\sqrt{n}}\right) &\rightrightarrows 
  \sin x,
\end{aligned}
\end{equation}  
where the double arrows denote uniform convergence on compact subsets
of the   complex plane as $n\to \infty$, see formulas 18.11.7 and 18.11.8 in \cite{NIST}.

\begin{proposition} \label{prop:MH}
  Let $\lambda$ be an even partition of length $r=2s$ and $P_n(x)$, 
	$n\in \mathbb N_\lambda$ the corresponding exceptional Hermite
  polynomials as defined in \eqref{eq:Pndef}.  We then have
	\begin{align} \label{eq:evenHM}
		\frac{(-1)^{n-\frac{|\lambda|}{2}} \sqrt{n \pi}}{2^{2n-|\lambda|+2r} (n-\frac{|\lambda|}{2} + \frac{r}{2})!} 
		P_{2n} \left(\frac{x}{2\sqrt{n}}\right) & \rightrightarrows H_{\lambda}(0) \cos x, \\
		\label{eq:oddHM}
		\frac{(-1)^{n-\frac{|\lambda|}{2}} \sqrt{\pi}}{2^{2n-|\lambda|+2r} (n-\frac{|\lambda|}{2} + \frac{r}{2})!}		
		P_{2n+1} \left(\frac{x}{2\sqrt{n}} \right) & \rightrightarrows H_{\lambda}(0) \sin x.
		\end{align}
\end{proposition}
\begin{proof}
  Expand the Wronskian determinant in \eqref{eq:Pndef} along the last
  column as
  \begin{align} \label{eq:Pnexpansion}
    P_{n} = H_{\lambda} H_{n-|\lambda|+r}^{(r)} + \sum_{j=0}^{r-1}  Q_j H_{n-|\lambda|+r}^{(j)}
		\end{align}
where $Q_0 = \Wr[H'_{\lambda_r},\ldots, H'_{\lambda_1+r-1}]$ and in general, each coefficient
$Q_j$ is a differential polynomial in $H_{\lambda_r}, \ldots, H_{\lambda_1+r-1}$
of degree
  \begin{equation} \label{eq:degQj} 
	\deg   Q_j = |\lambda|+j-r
	\end{equation}
that is independent of $n$.

Write 
\begin{equation} \label{eq:an}
	a_{2n}  = \frac{(-1)^n \sqrt{n \pi}}{2^{2n} n!}, \qquad
  a_{2n+1} = \frac{(-1)^n \sqrt{\pi}}{2^{2n+1} n!}.
	\end{equation}
Then, since \eqref{eq:HermiteMH} can be differentiated any number of times
by properties of uniform convergence of entire functions, we have for every non-negative integer $j$,
\begin{align} \label{eq:HermiteMHdiff}
  \frac{a_{2n}}{(2\sqrt{n})^j}  H_{2n}^{(j)} \left(\frac{x}{2\sqrt{n}}\right)  &\rightrightarrows 
  \left(\frac{d}{dx}\right)^j \cos x  \\
  \frac{a_{2n+1}}{(2\sqrt{n})^j}
  H_{2n+1}^{(j)} \left(\frac{x}{2\sqrt{n}}\right) &\rightrightarrows 
  \left(\frac{d}{dx} \right)^j \sin x
\end{align}
as $n \to \infty$, uniformly on compact subsets of the complex plane.
Thus, since $|\lambda|$ and $r$ are even,
\[ \frac{a_{2n-|\lambda|+r}}{(2 \sqrt{n})^r} H_{2n-|\lambda| + r}^{(r)} \left( \frac{x}{2\sqrt{n}} \right) \rightrightarrows 
		(-1)^{r/2} \cos x \]
while, for $j=0, \ldots, r-1$,
\[ \frac{a_{2n-|\lambda|+r}}{(2 \sqrt{n})^r} H_{2n-|\lambda| + r}^{(j)} \left( \frac{x}{2\sqrt{n}} \right) \rightrightarrows 
		0 \]
From this we see that the first term in \eqref{eq:Pnexpansion} is the dominant term as $n \to \infty$.
Using the expression \eqref{eq:an} for $a_{2n-|\lambda|+r}$ we  get \eqref{eq:evenHM} after some simplifications.

Relation \eqref{eq:oddHM}  is proved in an analogous fashion.
\end{proof}

Theorem \ref{thm:zerospacing} is an almost immediate consequence of Proposition \ref{prop:MH}.

\begin{proof}[Proof of Theorem \ref{thm:zerospacing}]
Recall that for $n \geq |\lambda|+\lambda_1$, the polynomial $P_n$ has exactly $n-|\lambda|$
simple real zeros, and $|\lambda|$ non-real zeros.

It follows from \eqref{eq:evenHM} and Hurwitz's theorem from complex analysis,
that those zeros of $P_{2n} \left(\frac{x}{2\sqrt{n}} \right)$ that do not tend to infinity,
tend to the zeros of $\cos x$ as $n \to \infty$, with each zero of $\cos x$ attracting exactly
one zero of $P_{2n} \left(\frac{x}{2\sqrt{n}} \right)$.
Thus the non-real zeros of $P_{2n}\left(\frac{x}{2\sqrt{n}}\right)$ tend to infinity, and the real zeros
tend to the zeros of $\cos x$, which gives \eqref{eq:evenzeros}.

The limit \eqref{eq:oddzeros} follows from \eqref{eq:oddHM} in a similar way.
\end{proof}


\section{Proof of Theorem \ref{thm:semicircle}} \label{sect:reg-zeros}

We start with a lemma.
\begin{lemma} 
\label{lem:lincombo}
Let $n \in \mathbb N_{\lambda}$. Then 
$P_n$ is a linear combination of $H_n, \ldots, H_{n-s}$ where $s = |\lambda|+r$.
\end{lemma}
\begin{proof}
From the expansion \eqref{eq:Pnexpansion} and the fact that $H_n^{(j)}$ is a multiple of 
$H_{n-j}$ we find
\begin{equation} \label{eq:Pnexpansion2} 
	P_n(x) = \sum_{j=0}^{r} \tilde{Q}_j(x) H_{n-j}(x),   
	\end{equation}
where $\tilde{Q}_j$ is a multiple of $Q_j$. Thus $\deg \tilde{Q}_j = |\lambda| + j-r$,
see \eqref{eq:degQj}.
If $k < n - |\lambda|-r$ then $\deg (x^k \tilde{Q}_j(x)) < n + j-2r \leq n-j$ for $j = 0,1, \ldots, r$.
By the orthogonality of the Hermite polynomials we then have
\[ \int_{-\infty}^{\infty} x^k \tilde{Q}_j(x)  H_{n-j}(x) e^{-x^2} dx = 0, \qquad j =0,1, \ldots, r \]
and then also by linearity and \eqref{eq:Pnexpansion2}
\[ \int_{-\infty}^{\infty} x^k P_n(x) e^{-x^2} dx  =0, \qquad k = 0,1, \ldots, n-|\lambda|-r-1. \]
This implies that $P_n$ is a linear combination of $H_n, \ldots, H_{n-|\lambda|-r}$
as claimed in the lemma.
\end{proof}

Let us recall the following known result \cite[Theorem 3.1]{BD}.
\begin{theorem}[Beardon, Driver]
\label{thm:BD}
Let $\{ \pi_n \}_{n=0}^\infty$ be orthogonal polynomials associated
with a positive measure.  Fix $0<s<n$ and let $c_{1,n}<\cdots
<c_{n,n}$ be the zeros of $\pi_n$, listed in increasing order.  Let
$P$ be a polynomial in the span of $\pi_s, \ldots, \pi_n$.  Then, at
least $s$ of the intervals $(c_k,c_{k+1})$, $1\leq k<n$ contain a zero
of $P$.
\end{theorem}
\begin{corollary}
  \label{cor:BD}
  Let $\lambda$ be an even partition of length $r$, and set
  $s=|\lambda|+r$. For $n>s$, let $c_{1,n}<\cdots <c_{n,n}$ be the
  zeros of 
  the classical Hermite polynomial $H_n$, listed in increasing
  order. Then, at least $n-s$ intervals $(c_{k}, c_{k+1})$ contain a
  zero of the exceptional Hermite polynomial $P_n$.
\end{corollary}
\noindent Thus as $n \to \infty$, at least $n-s$ of the zeros of $P_n$
follow the zeros of the Hermite polynomial.  

We are now able to give the proof of Theorem \ref{thm:semicircle}.

\begin{proof}[Proof of Theorem \ref{thm:semicircle}]
Let $f$ be a bounded continuous function  on $\mathbb R$.

The normalized counting
measure of the zeros of Hermite polynomials, scaled by the factor $\sqrt{2n}$,
converges weakly to $\frac{2}{\pi} \sqrt{1-x^2} dx$, see \cite{deift}. Thus
\[
\lim_{n\to\infty} \frac{1}{n}\sum_{k=1}^n
f\left(\frac{c_{k,n}}{\sqrt{2n}}\right) = \frac{2}{\pi} \int_{-1}^1 f(x) \sqrt{1-x^2} dx. \]
Then also, if  $\xi_{k,n} \in (c_{k,n}, c_{k+1,n})$ for every $k =1, \ldots, n-1$,
\begin{equation} \label{eq:sclimit2}
\lim_{n\to\infty} \frac{1}{n}\sum_{k=1}^{n-1}
f\left(\frac{\xi_{k,n}}{\sqrt{2n}}\right) = \frac{2}{\pi} \int_{-1}^1 f(x) \sqrt{1-x^2} dx.
\end{equation}

Let $n>s=|\lambda|+r$. By Corollary \ref{cor:BD}, we can take $\xi_{k,n} \in (c_{k,n}, c_{k+1,n})$
to be a zero of $P_n$ for every but at most $s$ values of $k$. If we drop these at most $s$ indices $k$
from the sum in \eqref{eq:sclimit2}, then it will not affect the limit, since $f$ is bounded.
Then we have a sum over at least $n-s$ real zeros  of $P_n$. Extending the sum by including the remaining
real zeros (if any) of $P_n$ does not affect the limit either, since their number is bounded by $s-|\lambda|$.
Thus we obtain the limit \eqref{eq:sclimit1}. 
\end{proof}

\section{Proof of Theorem \ref{thm:xzeros}}
\label{sect:exc-zeros}

The following residue
property is essential for the proof that follows.

\begin{proposition} 
  \label{prop:zerores}
  Let $\lambda$ be a partition and let $P_n$ be the polynomial 
	as defined in \eqref{eq:Pndef}. Then, the meromorphic
  function
  \begin{equation} \label{eq:Pn2} 
		x \mapsto  \frac{P_n(x)^2}{H_{\lambda}(x)^2}  e^{-x^2}, \qquad x\in    \mathbb C,
		\end{equation}
	has vanishing residues at each of its poles (which are the zeros of $H_{\lambda}$).
\end{proposition}
\noindent
Note that for $n \neq m$ the residues of $x \mapsto P_n(x) P_m(x) \frac{e^{-x^2}}{H_{\lambda}(x)^2}$
are zero as well, since by \eqref{eq:PnPm} this function is the derivative
of a meromorphic function.

\begin{proof}
We consider the partition $\mu = (\lambda_2 \geq \cdots \geq \lambda_r)$
that is obtained from $\lambda$ by dropping the first component.
We denote by $\{P_{m,\mu}\}_{m \in \mathbb N_{\mu}}$ the sequence of polynomials
associated with $\mu$.

From the definitions \eqref{eq:Hlambda} and \eqref{eq:Pndef} it is clear that $H_{\lambda} = P_{|\lambda|,\mu}$,
and thus, as already noted in the paragraph containing \eqref{eq:Doplambda},
\[ \psi(x) := \frac{H_{\lambda}(x)}{H_{\mu}(x)} e^{-\frac{1}{2} x^2}  \]
is an eigenfunction of the differential operator
\[ L_{\mu} =  - \frac{d^2}{dx^2}  + \left(x^2 - 2\log H_{\mu}(x) \right), \]
with eigenvalue $2 \lambda_1 + 1$.

Then by the properties of Darboux transformation, we have the factorization
\begin{equation} \label{eq:Darbouxmu} 
	L_{\mu} = A^{\dagger} A + (2 \lambda_1 + 1)
	\end{equation}
with
\begin{equation} \label{eq:AandAdag}
		A =  - \frac{d}{dx} +  \frac{\psi'(x)}{\psi(x)}, \quad \text{ and }  \quad
		A^{\dagger} = \frac{d}{dx} + \frac{\psi'(x)}{\psi(x)},
		\end{equation}
and 
\begin{align} 
	\label{eq:Darbouxlambda} 
	L_{\lambda} := - \frac{d^2}{dy^2}  + \left(x^2 - 2 \log H_{\lambda}(x) \right) 
  = A A^{\dagger} + (2\lambda_1-1). 
	\end{align}  
In addition, if $\varphi \neq \psi$ is any other eigenfunction of $L_{\mu}$,
then $A \varphi$ is an eigenfunction of $L_{\lambda}$, and every
eigenfunction of $L_{\lambda}$ is obtained this way.
Thus, associated with the polynomial $P_{n} = P_{n,\lambda}$ there is
an index $m$ such that
\begin{equation} \label{eq:Darbouxphi} 
	\varphi_{n, \lambda} = A \varphi_{m,\mu}, 
	\end{equation}
where
\begin{align*}
	\varphi_{n,\lambda}(x) = \frac{P_{n, \lambda}(x)}{H_{\lambda}(x)}  e^{-\frac{1}{2}x^2}, \qquad
	\varphi_{m,\mu}(x) = \frac{P_{m, \mu}(x)}{H_{\mu}(x)} e^{-\frac{1}{2} x^2}.
\end{align*}

Next, we get from \eqref{eq:AandAdag} that 
\begin{equation} \label{eq:fAg}  
	f (Ag) - (A^{\dagger} f) g = - \frac{d}{dx} (fg). 
	\end{equation}
Taking $f = \varphi_{n, \lambda}$, $g= \varphi_{m,\mu}$ in \eqref{eq:fAg}, and noting \eqref{eq:Darbouxphi}
and
\[ A^{\dagger} f = A^{\dagger} A g =  (L_{\mu} - (2\lambda_1+1)) g =  c g \]
for some constant $c$, since $g$ is an eigenfunction of $L_{\mu}$, we obtain
\begin{equation} \label{eq:Darbouxenergy} 
	\varphi_{n, \lambda}^2 -  c \varphi_{m, \mu}^2 = 
	- \frac{d}{dx} \left( \varphi_{n, \lambda} \varphi_{m, \mu} \right). \end{equation}
	
Now the proposition follows by induction on $r$. It is true if $r=0$, since then 
$H_{\lambda} \equiv 1$, and \eqref{eq:Pn2}  has no poles at all. Assuming 
the proposition is true for partitions of length $r-1$. Then the term
$c\varphi_{m, \mu}^2$ in \eqref{eq:Darbouxenergy} has zero residues at each of its poles. 
Also the right-hand side of \eqref{eq:Darbouxenergy} has
zero residues since it is the derivative of a meromorphic function. Thus
$\varphi_{n, \lambda}^2$ has zero residues at each of its poles, which proves
the proposition.	
\end{proof}

We are now ready for the proof of Theorem \ref{thm:xzeros}. 

\begin{proof}[Proof of Theorem \ref{thm:xzeros}]
Let $z_1, \ldots, z_{|\lambda|}$ be the zeros of $H_{\lambda}$ and assume
that $z_j$ is a simple zero of $H_{\lambda}$. Without loss of generality we may
assume that $\Im z_j > 0$.
By Proposition \ref{prop:zerores},
\[ \frac{P_n(x)^2}{\prod_{k \neq j} (x-z_k)^2} e^{-x^2}, \] 
then has a zero derivative at $x=z_j$. The logarithmic derivative is zero as
well, which implies
\begin{equation} \label{eq:logderzj} 
	2 \frac{P_n'(z_j)}{P_n(z_j)} - 2 \sum_{k\neq j} \frac{1}{z_j-z_k} - 2 z_j = 0. 
	\end{equation}
As before, let $x_{1,n}, \ldots, x_{n-|\lambda|,n}$  denote the real
zeros of $P_n$, and $z_{1,n},\ldots, z_{|\lambda|,n}$ the exceptional
zeros.  Then \eqref{eq:logderzj} tells us that
\begin{equation} \label{eq:zerobalance} 
	\sum_{k=1}^{n-|\lambda|} \frac{1}{z_j-x_{k,n}} +
\sum_{k=1}^{|\lambda|} \frac{1}{z_j - z_{k,n}} = z_j + \sum_{k \neq j} \frac{1}{z_j-z_k} 
\end{equation}


From  Plancherel-Rotach asymptotics of the Hermite polynomials, see e.g.\ \cite[Theorem 8.22.9]{Szego} 
one easily finds that for any compact interval $[a,b]$ on the real line,
the number of zeros of $H_n$ lying in $[a,b]$ grows roughly like $c \sqrt{n}$,
as $n \to \infty$, for some constant $c > 0$. By Corollary \ref{cor:BD}, the same
holds for the number of zeros of $P_n$ in such an interval.  

Any real zero $x_{k,n}$ in $[\Re z_j - 1, \Re z_j +1]$ has 
$|z_j- x_{k,n}|^2 \leq 1 + (\Im z_j)^2$ and so
\begin{equation} \label{eq:xknestimate} 
	\frac{1}{|z_j - x_{k,n}|^2} \geq \frac{1}{1 + (\Im z_j)^2}. 
	\end{equation}
Then for some constant $c_1 > 0$,
\[ \Im \left( \sum_{k=1}^{n-|\lambda|} \frac{1}{z_j-x_{k,n}} \right) = - \sum_{k=1}^n \frac{\Im z_j}{|z_j-x_{k,n}|^2}
	  < - c_1 \sqrt{n} \]
since all terms in the sum have the same sign and  at least  $c\sqrt{n}$ of them
satisfy \eqref{eq:xknestimate}.
The right-hand side of  \eqref{eq:zerobalance} does not depend on $n$, and so it follows that
for $n$ large enough, 
\[ \sum_{k=1}^{|\lambda|} \Im \left(  \frac{1}{z_j-z_{k,n}} \right) > c_1 \sqrt{n}. \] 
Since the number of terms does not depend on $n$,
at least one of the terms is of order $\sqrt{n}$.  Thus,
for $n$ sufficiently large, there is a non-real zero $z_{k,n}$ of $P_n$ with
\[ \Im \left( \frac{1}{z_j-z_{k,n}} \right) > c_2 \sqrt{n}, \qquad c_2 = \frac{c_1}{|\lambda|}. \]
This is easily seen to imply that for this $k$,
\[ |z_j - z_{k,n}| \leq \frac{1}{c_2 \sqrt{n}}. \] 
In fact, $\Im z_{k,n} > \Im z_j$ and $z_{k,n}$ lies in a circle centred at $z_j +
\frac{i}{2c_2 \sqrt{n}}$ with radius $\frac{1}{2c_2 \sqrt{n}}$.
Theorem \ref{thm:xzeros} is proved.
\end{proof}

\section*{Acknowledgements}
The first author is supported by KU Leuven Research Grant OT/12/073,
the Belgian Interuniversity Attraction Pole P07/18, and FWO Flanders
projects G.0641.11 and G.0934.13.  The second author is supported by
NSERC grant RGPIN-228057-2004.




\end{document}